\documentclass[reqno,12pt]{amsart} %
\usepackage{epsfig}  
\usepackage{amssymb,amsmath}   
\usepackage{graphicx}     

\usepackage[paperwidth=8.5in,paperheight=11.0in,
  left=1.5in,right=1.0in,top=1.0in,bottom=1.0in]{geometry}

\begin{document}

\makeatletter
\def\rank{\mathop{\operator@font rank}\nolimits}
\makeatother

\newtheorem{thm}{Theorem}[section]
\newtheorem{ex}[thm]{Example}
\newtheorem{lem}[thm]{Lemma}
\newtheorem{rmk}[thm]{Remark}
\newtheorem{defi}[thm]{Definition}

\title[Local null controllability]{Local null controllability of the control-affine nonlinear systems with time-varying disturbances. Direct calculation of the null controllable region}

\author{Robert Vrabel} 
\address{Robert Vrabel, Slovak University of Technology in Bratislava, Faculty of Materials Science and Technology, Institute of Applied Informatics, Automation and Mechatronics,  Bottova~25,  917 01 Trnava, Slovakia}
\email{robert.vrabel@stuba.sk}

\date{{\bf\today}}

\begin{abstract}
The problem of local null controllability for the control-affine nonlinear systems $\dot x(t)=f(x(t))+Bu(t)+w(t),$ $t\in[0,T]$ is considered in this paper. The principal requirements on the system are that the LTI pair $\left((\partial f/\partial x)(0), B\right)$ is controllable and the disturbance is limited by the constraint $|f(0)+w(t)|\leq M_d\left(1-\frac{t}{T}\right)^\eta,$  $M_d\geq0$ and $\eta>0.$ These properties together with one technical assumption yield a complete answer to the problem of deciding when the null controllable region have a nonempty interior. The criteria obtained involve purely algebraic manipulations of vector field $f,$ input matrix $B$  and bound on the disturbance $w(t).$ To prove the main result we have derived a new Gronwall-type inequality allowing the fine estimates of the closed-loop solutions. The theory is illustrated and the efficacy of proposed controller is demonstrated by the examples where the null controllable region is explicitly calculated. Finally we established the sufficient conditions to be the system under consideration (with $w(t)\equiv 0$) globally null controllable.
\end{abstract}

\keywords{Nonlinear systems, null controllability, null controllable region, state feedback, Gronwall-type inequality.}
\subjclass[2000]{93B05, 93C10, 93D15}

\maketitle

\section[Introduction]{Introduction}

In this paper we will concern ourselves with the problem of null controllability of the nonlinear systems of the form 
\begin{equation}\label{def_system}
\dot x(t)=f(x(t))+Bu(t)+w(t),\ t\in[0,T], \ x(0)=x_0\in\mathbb{R}^n, 
\end{equation}
where $x(t)\in\mathbb{R}^n$ is the state variable, $u(t)\in\mathbb{R}^m$ is the control input and $w(t)\in\mathbb{R}^n$ represents the total disturbance (unmodelled system dynamics, uncertainty, overall external disturbances that affect the system, {\it etc.}) which is  potentially unknown but with known magnitude constraint $|f(0)+w(t)|\leq M_d\left(1-\frac{t}{T}\right)^\eta$ for some $T,$ $M_d\geq0$ and $\eta>0,$ specified below in the proof of Lemma~\ref{lemma1}. The function $f$ is $C^2$ on $\mathbb{R}^n$ and $B$ is an $n\times m$ constant matrix.  We establish the sufficient conditions for the existence and the method for determining null controllable region $\mathcal{X}_0$ in the sense of the definitions below. Henceforth, we use the following notations: The $n\times n$ matrix $(\partial f/\partial x)(0)$ is a Jacobian matrix of the vector field $f(x)$ evaluated at $x=0$ and an upper dot indicates a time derivative. The superscript $'\,T\,'$ is used to indicate transpose operator. We denote by $|\cdot|$ the Euclidean norm and by $||\cdot||$ a matrix norm induced by the Euclidean norm of vectors, $||A||=\max_{|x|=1}|Ax|.$ It is well-known, see e.~g. \cite{Rugh}, that this norm is equivalent to the spectral norm for matrices, $||A||=\sqrt{\lambda_{max}(A^TA)}.$  The real part of a complex number $z$ is denoted by $\Re(z).$  Further, we shall always assume that the domain of existence of trajectories for the control system (\ref{def_system}) is at least the interval $[0,T]$ for every $x_0,$ every continuous input $u(t)$ and the continuous disturbance $w(t)$ satisfying the constraint above. 

The properties of the systems related to controllability have been analyzed by many researchers for different meaning, among them the concept of null controllability. We will use the following definition from \cite{FabbriJonKlo} that we modified for our purposes. 
\begin{defi}\label{locally_null_c1}  
The system (\ref{def_system}) is said to be {\bf locally null controllable} if there exists an open neighborhood $\mathcal{X}_0$ of the origin in $\mathbb{R}^n$ and a finite time $T>0$  such that, to each $x_0\in\mathcal{X}_0,$ there corresponds a continuous  function $u: [0,T]\rightarrow \mathbb{R}^m$ such that the solution $x(t)$ of (\ref{def_system}) determined by this $u=u(t)$ and $x(0)=x_0$ satisfies $x(T)=0.$
\end{defi}

In general, the concept of controllability is defined as an open-loop control, but in many situations a state feedback control is preferable. The definition is as follows.
\begin{defi}\label{locally_null_c2}   
The system (\ref{def_system}) is said to be {\bf locally null controllable by a state feedback controller} if there exists an open neighborhood $\mathcal{X}_0$ of the origin in $\mathbb{R}^n$ and a finite time $T>0$  such that, to each $x_0\in\mathcal{X}_0,$ there corresponds a continuous control law $v(x,t)$ such that the solution $x(t)$ of (\ref{def_system}) determined by $u=v(x,t)$ and $x(0)=x_0$ satisfies $x(T)=0.$
\end{defi} 

The majority of results for controllability have been established for linear time-invariant (LTI) systems $\dot x=Ax+Bu,$ where Kalman \cite{Kal1}, \cite{Kal2} has shown that a necessary and sufficient condition for global ($\mathcal{X}_0=\mathbb{R}^n$) null controllability  is 
\begin{equation*}
\rank\left(B, AB, \dots,  A^{n-1}B\right)=n,
\end{equation*}
or equivalently, the controllability Gramian matrix $W_c(t_1),$
\begin{equation*}
W_c(t_1)\triangleq\int\limits_0^{t_1} P(t)BB^TP^T(t)dt,
\end{equation*} 
is invertible for any $t_1>0.$ Here $P(t)=e^{At}$ is a fundamental matrix of homogeneous system $\dot x=Ax.$  

The situation is more delicate for linear time-variant (LTV) systems $\dot x=A(t)x+B(t)u.$ Although there is a well-known Gramian matrix-based criterion for the (global) null controllability of such systems, but for a general LTV system there is no analytical expression that expresses $P(t)$ as a function of $A(t).$ Nevertheless, in  \cite{Benzaid} (Theorem~2.~1) has been proved that small perturbations $V(t)$ and $B(t)$ of constant matrices $A$ and $B,$ respectively, preserves the (global) null controllability. 

In terms of nonlinear system controllability, one of the most important results in this field was derived by Lee and Markus \cite{LeeMarkus}, \cite{MarkusLee}. The result states that if a linearized system $\dot x=Ax+Bu$ at an equilibrium point $(0,0)$ [$A=(\partial f/\partial x)(0,0)$ and $B=(\partial f/\partial u)(0,0)$] is controllable, then there exists a local controllable area of the original nonlinear system $\dot x=f(x,u)$ around this equilibrium point. 

Later, it turned out that that fact is also true in the case when the linearized system is time-varying \cite[p.~127]{Coron} and this result gives us a good reason to locally use linearized system instead of the original nonlinear system. In particular, this applies when {\bf (i)} the system is linearized around equilibrium point, in which case the matrices $A,$ $B$ are constants, and the controllability of LTI systems is easy to verify, or 
{\bf (ii)} the results of global controllability do not hold or are not easy to be obtained. The drawback of this approach is that the fundamental theorems do not refer on the region where we can use the linearized systems instead of the original nonlinear systems. Some of the few papers concerning with this topic are \cite{KbayashiShim} and \cite{MahmoudMhaskar}, or \cite{HuLinQiu} for LTI systems with a constrained input.

Completely different principles and techniques than those based on the linearization around the trajectory of control system are behind the geometric control theory. This theory,  for the time-invariant systems $\dot x = f(x,u),$ establishes a connection between the Lie algebras of vector fields and the sets of points reachable by following flows of vector fields.  For your reference, see e.~g. the pioneering works \cite{Brockett}, \cite{Brunovsky}, \cite{chow}, \cite{Lobry1}, \cite{Lobry2} and \cite{Sussmann} or the now classical monographs  \cite{Isidori}, \cite{NijmeijerSchaft}, \cite{Sontag}. The standard assumption that is made throughout these works is that $f$ is an analytic function of the variable $x.$ This analyticity assumption cannot be relaxed without destroying the theory as was carefully analyzed and emphasized in \cite{SussmannJurdjevic}.

Unfortunately, none of this theories is not applicable to the systems considered in present paper in general, and to the best of our knowledge, there has been probably very limited (if any) research on null controllability of nonlinear systems with time-varying disturbances, and thus, this topic does not seem to have been well studied until now. Moreover, the technique of the proof of Lemma~\ref{lemma1} (Section~\ref{lemmas_thm}) allows us to explicitly estimate the null controllable region (Example~\ref{example1}).

In the following section, we formulate the technical result, the new Gronwall-type inequality, used in the proof of Lemma~\ref{lemma1} providing a quantitative estimate of the solutions to system~(\ref{def_system}) on the interval $[0,T].$

\section{Technical result: Gronwall-type inequality}
\begin{lem}(Compare with \cite[p.~35]{Bellman})\label{Gronwall}
If $u_1,v_1,w_1\geq0,$ if $c_1$ is a positive constant, and if
\begin{equation}\label{gronwall1}
u_1\leq c_1+\int\limits_0^t(u_1v_1+w_1)dt_1
\end{equation}
then
\begin{equation}\label{gronwall2}
u_1\leq e^{\int_0^tv_1dt_1}\left[c_1+\rho\left(e^{\int_0^t\frac{w_1}\rho dt_1}-1\right)\right],\ \forall t\geq0,\ \forall\rho>0.
\end{equation}
\end{lem}
\begin{proof}
From (\ref{gronwall1}) and the inequality $e^z\geq z+1$ for $z=\int_0^t\frac {w_1}\rho\geq0$
we have for all $\rho>0$
\begin{equation}\label{gronwall3}
u_1\leq c_1+\rho e^{\int_0^t\frac {w_1}\rho dt_1}-\rho+\int_0^tu_1v_1dt_1
\end{equation}
which implies
\[
\frac{u_1}{c_1+\rho\left(e^{\int_0^t\frac {w_1}\rho dt_1}-1\right)+\int_0^tu_1v_1dt_1}\leq1.
\]
Multiplying this with $v_1\geq0$ we obtain
\[
\frac{u_1v_1}{c_1+\rho\left(e^{\int_0^t\frac {w_1}\rho dt_1}-1\right)+\int_0^tu_1v_1dt_1}\leq v_1
\]
and
\[
\frac{u_1v_1+w_1e^{\int_0^1\frac {w_1}\rho dt_1}}{c_1+\rho\left(e^{\int_0^t\frac {w_1}\rho dt_1}-1\right)+\int_0^tu_1v_1dt_1}\leq v_1+\frac{w_1e^{\int_0^t\frac {w_1}\rho dt_1}}{c_1+\rho\left(e^{\int_0^t\frac {w_1}\rho dt_1}-1\right)}.
\]
Integrating both sides between $0$ and $t$ we have
\[
\ln\left[c_1+\rho\left(e^{\int_0^t\frac {w_1}\rho dt_1}-1\right)+\int_0^tu_1v_1dt_1\right]-\ln c_1
\]
\[
\leq\int_0^tv_1dt_1+\ln\left[c_1+\rho\left(e^{\int_0^t\frac {w_1}\rho dt_1}-1\right)\right]-\ln c_1.
\]
Converting this to exponential form  and taking into considerations (\ref{gronwall3}), the inequality becomes (\ref{gronwall2}).
\end{proof}

\section{The auxiliary lemmas and main result}\label{lemmas_thm}

\subsection{Part~I:~The existence of controller from Definition~\ref{locally_null_c2} for (\ref{def_system})}

To prove the existence of the controller we will proceed as follows. By using the time rescaling $t(\tau)=T\left(1-e^{-\omega\tau}\right),$ the original problem
\[
\frac{dx(t)}{dt}=f(x(t))+Bu(t)+w(t), \ x(0)=x_0,\ x(T)=0
\]
is transformed to the problem of asymptotic stabilizability of
\begin{equation}\label{transformed problem}
\frac{da(\tau)}{d\tau}=T\omega e^{-\omega\tau}\left(f(0)+Aa(\tau)+Bb(\tau)+R_2(a)+d(\tau)\right),\  a(0)=x_0, 
\end{equation}
where $a(\tau)=x(t(\tau)),$ $b(\tau)=u(t(\tau))$ and $d(\tau)=w(t(\tau)).$ The matrix $A=(\partial f/\partial x)(0)$ is the Jacobian matrix and $R_2(a)$ denotes a Taylor remainder. Let us now substitute 
\begin{equation}\label{controller_b}
b(\tau)=\frac1{T\omega}e^{\omega\tau}Ka(\tau)
\end{equation}
into the equation (\ref{transformed problem}); here $K$ represents an $m\times n$ constant gain matrix. We obtain 
\[
\frac{da}{d\tau}=\underbrace{(A+BK)}_{\triangleq A_{cl}}a+\underbrace{T\omega e^{-\omega\tau}A}_{\triangleq\tilde B(\tau)}a
\]
\[
+\underbrace{T\omega e^{-\omega\tau}R_2(a)-Aa+T\omega e^{-\omega\tau}(f(0)+d(\tau))}_{\triangleq\tilde R_\omega(a,\tau)},
\]
\[
=\left(A_{cl}+\tilde B(\tau)\right)a+\tilde R_\omega(a,\tau).
\]
It is easy to check that
\begin{equation*}
\int\limits_0^\infty||\tilde B(\tau)||d\tau=T||A||<\infty.
\end{equation*}
Since $R_2(a)$ is clearly $O(|a|^2)$ as $a\rightarrow 0$ we can find the number $\varepsilon>0$ and the constant $\Gamma_0>||A||$ such that
\begin{equation}\label{estimate_tR}
|\tilde R_\omega(a,\tau)|\leq\Gamma_0|a|+T\omega e^{-\omega\tau}|f(0)+d(\tau)|\ \mathrm{for}\ |a|\leq\varepsilon.
\end{equation}
The pair $(A,B)$ is assumed to be controllable therefore 
\begin{equation}\label{coefficient_k1}
||e^{A_{cl}\tau}||\leq k_1e^{-\tilde\lambda(K)\tau}\ \mathrm{for}\ \tau\geq0,
\end{equation}
where
\begin{equation}\label{tlambda}
\tilde\lambda(K)\triangleq-\max\left\{\Re\left(\lambda_i\right),\ i=1,\dots,n: \lambda_i\ \mathrm{be\ an\ eigenvalue\ of}\ A_{cl}\right\}
\end{equation}
and $k_1=k_1(\lambda_1,\dots,\lambda_n)$ is no less than unity (\cite[p.~101]{Rugh}). The coefficient $\tilde\lambda(K)$ may be chosen provisionally arbitrarily, but with all $\Re\left(\lambda_i\right)<0$ and $\lambda_i\neq\lambda_j$ for $i\neq j.$ The admissible range of the values $\tilde\lambda(K)$ is given in (\ref{range_tlambda}) below. Now we consider an auxiliary system
\begin{equation}\label{auxiliary_system}
\frac{d\tilde a}{d\tau}=A_{cl}\tilde a+\tilde B(\tau)\tilde a
\end{equation}
and its solution obtained by variation of parameters
\[
\tilde a(\tau)=e^{A_{cl}\tau}\tilde a(0)+\int\limits_0^\tau e^{A_{cl}(\tau-s)}\tilde B(s)\tilde a(s)ds.
\]
Hence we have an estimate
\[
|\tilde a(\tau)|\leq k_1e^{-\tilde\lambda\tau}|\tilde a(0)|+\int\limits_0^\tau k_1e^{-\tilde\lambda(\tau-s)}||\tilde B(s)|||\tilde a(s)|ds.
\]
Multiplying this with $e^{\tilde\lambda\tau}$ and substituting $|\tilde a(\tau)|e^{\tilde\lambda\tau}$ by $u_1(\tau)$ we get the following inequality for $u_1:$
\[
u_1(\tau)\leq k_1u_1(0)+\int\limits_0^\tau k_1u_1(s)||\tilde B(s)||ds.
\]
Using Lemma~\ref{Gronwall} with $w_1\equiv0,$ that is, the classical Gronwall inequality, we obtain
\[
u_1(\tau)\leq k_1u_1(0)e^{\int\limits_0^\tau k_1||\tilde B(s)||ds}
\]
or
\[
|\tilde a(\tau)|\leq k_1|\tilde a(0)|e^{-\lambda^*\tau} \ \mathrm{where}\ -\lambda^*\triangleq k_1T||A||-\tilde\lambda.
\]
Let $\Phi(\tau,s)=P(\tau)P^{-1}(s),$  where $P$ is a fundamental matrix of (\ref{auxiliary_system}), is its state transition matrix. As was proved in \cite[p.~102]{Rugh},
\begin{equation}\label{estimate_stm}
||\Phi(\tau,s)||\leq k_1e^{-\lambda^*(\tau-s)}
\end{equation}
for all $\tau,$ $s$ such that $\tau\geq s.$ Thus the solution $a(\tau)$ of (\ref{transformed problem}) may be expressed as
\[
a(\tau)=\Phi(\tau,0)a(0)+\int\limits_0^\tau\Phi(\tau,s)\tilde R_\omega(a(s),s)ds,
\]
and taking into consideration (\ref{estimate_tR}) and (\ref{estimate_stm}), we obtain an estimate
\[
|a(\tau)|e^{\lambda^*\tau}\leq k_1|a(0)|+\int\limits_0^\tau \left(k_1\Gamma_0e^{\lambda^*s}|a(s)|+k_1T\omega e^{(\lambda^*-\omega)s}|f(0)+d(s)|\right)ds.
\] 
By applying Lemma~\ref{Gronwall} with  $c_1=k_1|a(0)|=k_1|x_0|,$ $u_1=|a(\tau)|e^{\lambda^*\tau},$ $v_1=k_1\Gamma_0,$ $w_1=k_1T\omega e^{(\lambda^*-\omega)\tau}|f(0)+d(\tau)|$ and $\rho=$product of those parameters from the set $\{k_1, T, \omega\}$ that are greater than $1,$ and multiplying the result by $e^{-\lambda^*\tau}$ we have
\[
|a(\tau)|\leq e^{-\lambda^{**}\tau}\left[k_1|x_0|+\omega\left(e^{k_1T\int_0^\tau e^{(\lambda^{**}+k_1\Gamma_0-\omega)s}|f(0)+d(s)|ds}-1\right)\right],
\]
where $-\lambda^{**}\triangleq k_1\Gamma_0-\lambda^*.$ Taking into account the example illustrating the theory and without loss of generality we have used in the previous step $\rho=\omega,$ which is the right choice for $k_1=1,$ $T=1$ and $\omega=2.$ At this point,  it turns out one of the benefits of this variant of Gronwall-type inequality, namely, we can manipulate with the parameters $\omega,$ $k_1$ and $T$ to be not in the exponent, which can be especially relevant in the calculation of (as large as possible radius $\mu$ of the ball contained in) the null controllable region, see the equation (\ref{ball}) below.  The above-described simple rule for determining the parameters from $\{k_1, T, \omega\}$ that should remain or be removed from the exponent follows directly by comparing the graphs of the functions $y_1(\xi)=e^{\alpha\xi}-1$ and $y_2(\xi)=\left(e^{\alpha}-1\right)\xi$ on the interval $[0,\infty)$ for an arbitrary but fixed parameter $\alpha>0,$  where we find out that $y_1(\xi)-y_2(\xi)<0$ for $\xi\in(0,1)$ and $y_1(\xi)-y_2(\xi)>0$ for $\xi\in(1,\infty)$ independently on the value of $\alpha.$ Now, if   
\begin{equation}\label{estimate_dist}
|f(0)+d(\tau)|\leq M_de^{\delta\tau},\ M_d\geq0, \ \delta\in\mathbb{R}, \tau\geq0,
\end{equation} 
then
\begin{equation}\label{estimate_sol}
|a(\tau)|\leq e^{-\lambda^{**}\tau}\left[k_1|x_0|+\omega\left(e^{k_1TM_d\int_0^\tau e^{(\lambda^{**}+k_1\Gamma_0-\omega+\delta)s}ds}-1\right)\right].
\end{equation}
Hence the solution $a(\tau)\rightarrow 0$ for $\tau\rightarrow \infty$ if 
\begin{itemize}
\item[] $\lambda^{**}>0$ which is equivalent to $\tilde\lambda>k_1\left(\Gamma_0+T||A||\right)$ 
\end{itemize}
and at the same time, if
\begin{itemize}
\item[] $\lambda^{**}+k_1\Gamma_0-\omega+\delta<0,$  that is, $\tilde\lambda<k_1T||A||+\omega-\delta.$
\end{itemize}
Thus, $\tilde\lambda=\tilde\lambda(K)$ in (\ref{tlambda}) must satisfy 
\begin{equation}\label{range_tlambda}
k_1(\Gamma_0+T||A||)<\tilde\lambda<k_1T||A||+\omega-\delta,
\end{equation}
where the parameter  $\omega$ is such that this inequality has meaning. Clearly, for $M_d=0$ this inequality reduced to only its left-hand side. Thus, we get the final estimate of $a(\tau)$ in the form   
\begin{equation}\label{final_estimate2}
|a(\tau)|\leq e^{-\lambda^{**}\tau}\left[k_1|x_0|+\left(e^{-\frac{k_1TM_d}{\gamma}}-1\right)\omega\right]\ \mathrm{for\ all}\ \tau\geq 0,
\end{equation}
where
\[
\lambda^{**}=\tilde\lambda-k_1\left(\Gamma_0+T||A||\right)>0, \quad  \gamma=\lambda^{**}+k_1\Gamma_0-\omega+\delta<0,
\]
and $|x_0|$ is such that for all $\tau\geq0$  is $|a(\tau)|\leq\varepsilon;$ the parameter $\varepsilon$ is defined in (\ref{estimate_tR}). Analyzing (\ref{final_estimate2}), this is satisfied if
\begin{equation}\label{estimate_on_x0} 
k_1|x_0|+\left(e^{-\frac{k_1TM_d}{\gamma}}-1\right)\omega\leq\varepsilon.
\end{equation}
So $\mathcal{X}_0$ contains an open ball $\left\{x_0\in\mathbb{R}^n:\ |x_0|<\mu\right\},$ where $\mu>0$ is a solution of the equation
\begin{equation}\label{ball}
k_1\mu+\left(e^{-\frac{k_1TM_d}{\gamma}}-1\right)\omega=\varepsilon.
\end{equation}
Now, the estimate for solution $x(t)$ of the original system (\ref{def_system}) we obtain by backward substitution $\tau=\frac1\omega\ln\left(\frac{T}{T-t}\right)$ in (\ref{final_estimate2}):
\begin{equation}\label{final_estimate3}
|x(t)|\leq\left(1-\frac{t}{T}\right)^{\lambda^{**}/\omega}\left[k_1|x_0|+\left(e^{-\frac{k_1TM_d}{\gamma}}-1\right)\omega\right].
\end{equation}
Because this inequality holds for every $x_0\in\mathbb{R}^n$ satisfying $|x_0|\leq\mu$ the system is locally null controllable by continuous state feedback controller 
\begin{equation}\label{controller}
u=v_\omega(x,t)=\left(\frac{1}{T-t}\right)\frac{Kx}{\omega}, \ t\in[0,T],
\end{equation}
obtained from (\ref{controller_b}) by the same substitution as above. 

\subsection{Part~II:~Boundedness of controller} 
The controllers defined by the equality (\ref{controller}) could be potentially unbounded in the left neighborhood of $t=T.$ In this part, we determine the sufficient conditions guaranteeing the boundedness of this controller on the whole interval $[0,T].$ Obviously, from the inequality (\ref{final_estimate3}) and (\ref{controller}) follows that this is satisfies if $\lambda^{**}/\omega\geq1.$  From definition of the coefficient $\lambda^{**}$ and (\ref{range_tlambda}) we have that $\lambda^{**}\in(0,\omega-\delta-k_1\Gamma_0)$ for $M_d>0$  and $\lambda^{**}\in(0,\infty)$ for $M_d=0$ because the expression in the parentheses of the estimate (\ref{estimate_sol}) is identically zero in the second case. Hence $\lambda^{**}/\omega\in(0,1-\frac{\delta+k_1\Gamma_0}{\omega})$ or $\lambda^{**}/\omega\in(0,\infty),$ respectively, so that the necessary condition to be $\lambda^{**}/\omega\geq1$ is $\delta<-k_1\Gamma_0,$ and (\ref{estimate_dist}) implies the inequality $|f(0)+w(t)|\leq M_d\left(1-\frac{t}{T}\right)^{-\delta/\omega}$ for $t\in[0,T].$  

\centerline{}

Summarizing the above findings, we reach the following conclusion. 
\begin{lem}\label{lemma1}
Let us consider the control system (\ref{def_system}). Assume that
\begin{itemize}
\item[(H1)]the LTI pair $\left(A, B\right),$  $A=(\partial f/\partial x)(0),$ is controllable; 
\item[(H2)] there exist the constants $\tilde\lambda>0,$ $\omega>0,$ $\delta<0$ and time $T$ such that
\begin{equation}\label{hypothesisH2_1}
k_1(\Gamma_0+T||A||)+\omega\leq\tilde\lambda\leq k_1T||A||+\omega-\delta \ (\mathrm{for}\ M_d>0)
\end{equation}
or
\[
k_1(\Gamma_0+T||A||)+\omega\leq\tilde\lambda \ (\mathrm{for}\ M_d=0);
\]
\item[(H3)] $|f(0)+w(t)|\leq M_d\left(1-\frac{t}{T}\right)^{-\delta/\omega},$ $t\in[0,T],$ where $\delta<-k_1\Gamma_0$ for $M_d>0$ and  $\delta=0$ if $M_d=0;$
and 
\item[(H4)] the equation (\ref{ball}) has a positive solution $\mu.$

\end{itemize}
Then the system (\ref{def_system}) is locally null controllable by a bounded, continuous on $[0,T]$ state feedback controller of the form 
\begin{equation*}
u=v_\omega(x,t)=\left(\frac{1}{T-t}\right)\frac{Kx}{\omega}, \ t\in[0,T].
\end{equation*}
The set $\mathcal{X}_0$ from Definition~\ref{locally_null_c2} contains an open ball $\left\{x_0\in\mathbb{R}^n:\ |x_0|<\mu\right\}.$ 
The constants  $\Gamma_0,$ $k_1$ and $\tilde\lambda=\tilde\lambda(K)$ are defined in  (\ref{estimate_tR}), (\ref{coefficient_k1}) and (\ref{tlambda}), respectively.
\end{lem}
\begin{rmk}
\rm In the connection with the assumption (H2), it is worth noting that if the input matrix $B$ is regular (i.~e., $m=n$ and $\mathrm{det}(B)\neq0$), then $K=B^{-1}(\Delta-A),$ where the state matrix of the closed-loop system $A_{cl}=\Delta$ is diagonal matrix $\mathrm{diag}(\lambda_1,\lambda_2)$ with arbitrary $\lambda_1<\lambda_2<0$ and in this case it is easy to ensure the fulfillment of the left inequality in (\ref{hypothesisH2_1}) (in this case $k_1=1;$ for the details see Example~\ref{example1}).
\end{rmk}
Now we introduce and prove the statement connecting Lemma~\ref{lemma1} with the main result of present paper, Theorem~\ref{main_result}.
\begin{lem}\label{lemma2}
If the system (\ref{def_system}) is locally null controllable by a state feedback controller then the system (\ref{def_system}) is locally null controllable.
\end{lem}
\begin{proof}
The proof is straightforward. If the system (\ref{def_system}) is locally null controllable by a state feedback controller, it is locally null controllable by the input determined by this control law. 
\end{proof}

Now the main result may be formulated as follows. 

\begin{thm}\label{main_result}
Under the assumptions (H1)-(H4) of Lemma~\ref{lemma1} the system (\ref{def_system}) is locally null controllable.
\end{thm}
The statement immediately follows from the Lemma~\ref{lemma1} and Lemma~\ref{lemma2}. 

\section[null controllable region]{Explicit calculation of the null controllable region}

The applicability of our approach to the explicit calculation of null controllable region is illustrated on the following example.

\begin{ex}\label{example1}
\rm Consider the nonlinear control system
\begin{equation}\label{example_calc}
{\setlength\arraycolsep{2pt}
\left(
\begin{array}{c}
  \dot x_1  \\
  \dot x_2 
\end{array} 
\right)=
\left(
\begin{array}{c}
  \frac{x_2}{1+x_2^2} \\
   x_1^2
\end{array} 
\right)+
\left(
\begin{array}{rc}
  1 & 1 \\
  -1 & 3 
\end{array} 
\right)
\left(
\begin{array}{c}
  x_1  \\
  x_2 
\end{array} 
\right)+
\left(
\begin{array}{rc}
  1 & 0 \\
  0 & 1 
\end{array} 
\right)
\left(
\begin{array}{c}
  u_1  \\
  u_2 
\end{array} 
\right)
}
\end{equation}
\[
\triangleq 
{\setlength\arraycolsep{2pt}
\left(
\begin{array}{c}
  f_1(x_2) \\
  f_2(x_1) 
\end{array} 
\right)+
\tilde A
\left(
\begin{array}{c}
  x_1  \\
  x_2 
\end{array} 
\right)+
B
\left(
\begin{array}{c}
  u_1  \\
  u_2 
\end{array} 
\right).
}
\]
Linearizing this system at $x=0$ we have
\begin{equation*}
{\setlength\arraycolsep{2pt}
A=\left(\partial f/\partial x\right)(0)=\left(\begin{array}{cc}
 0 \ & 1   \\
 0 \ & 0
\end{array} \right)+\tilde A\ \mathrm{and}\ 
B=\mathrm{id} 
}.
\end{equation*}
Therefore  the linearized system is controllable.  Now we compute the constants $\Gamma_0$ and $\varepsilon$ in (\ref{estimate_tR}). Clearly 
\[
|R_2(a_1,a_2)|\leq\frac12\max_{a_2\in\mathbb{R}} |f_1''(a_2)|a_2^2 +\frac12\max_{a_1\in\mathbb{R}} |f_2''(a_1)|a_1^2
\]
\[
\leq 0.72855a_2^2+a_1^2\leq1.72855|a|^2.
\]
So, taking into account that spectral norm $||A||=\sqrt{\lambda_{\max}(A^TA)}=3.6180,$ 
\[
\left\vert T\omega e^{-\omega\tau}R_2(a)-Aa\right\vert\leq1.72855T\omega|a|^2+3.6180|a|\leq\Gamma_0|a|.
\]
The last inequality is satisfied for 
\begin{equation}\label{example_ineq}
|a|\leq\varepsilon(T,\Gamma_0,\omega)=\frac{\Gamma_0-3.6180}{1.72855 T\omega},\ \ \Gamma_0>3.6180.
\end{equation}
Further, let us consider only the real eigenvalues $\lambda_1,\lambda_2$  of $A_{cl}=A+BK$ ($\lambda_1<\lambda_2<0$). From the properties of spectral norm we get
\[
||e^{A_{cl}\tau}||=||e^{PJP^{-1}\tau}||=||Pe^{J\tau}P^{-1}||\leq||P||||P^{-1}||||e^{J\tau}||=||P||||P^{-1}||e^{\lambda_2\tau},
\]
where $J$ is a Jordan canonical form of the matrix $A_{cl}$ obtained by some similarity transformation $P,$ $J=P^{-1}A_{cl}P.$ 

For the state feedback gain matrix 
\[
{\setlength\arraycolsep{2pt}
K=\left(
\begin{array}{rr}
 -12 \ & -2   \\
 1 \ & -13
\end{array} 
\right)
}
\]
we obtain 
\[
{\setlength\arraycolsep{2pt}
A_{cl}=\left(
\begin{array}{rr}
 -11 \ & 0   \\
 0 \ & -10
\end{array} 
\right)
}
\]
directly in the Jordan canonical form, therefore $P=\mathrm{id}$ and $k_1$ from (\ref{coefficient_k1}) can be chosen as $k_1=||P||||P^{-1}||=1.$ Now, for example, if $\Gamma_0=4$ and $T=1,$ then the existence and boundedness of the controller (\ref{controller}) is guaranteed for $\omega$ satisfying
\[
0<\omega\leq\lambda^{**}=-\lambda_2-k_1(\Gamma_0+T||A||),\ \mathrm{i.~e. \ for}\ 0<\omega\leq 2.382.
\]
The corresponding state feedback controller is 
\begin{equation}\label{controller_example}
{\setlength\arraycolsep{2pt}
\left(
\begin{array}{c}
  u_1(t) \\
  u_2(t) 
\end{array} 
\right)=\left(\frac{1}{1-t}\right)\frac{
\left(
\begin{array}{rr}
 -12 \ & -2   \\
 1 \ & -13
\end{array} 
\right)
(x_1(t),x_2(t))^T}{\omega}, \ t\in[0,1],
}
\end{equation}
and $\omega\in(0,2.382].$ For example, if $\omega=2,$ the null controllable region $\mathcal{X}_0$ contains the open ball 
\[
|x_0|<\varepsilon/k_1=0.110497/1=0.110497
\]
and we have the estimate on the solution of (\ref{def_system}) with the initial state satisfying $|x_0|<0.110497$ in the form
\[
|x(t)|\leq\varepsilon\left(1-\frac{t}{T}\right)^{\lambda^{**}/\omega}=0.110497\left(1-t\right)^{2.382/2},
\]
as follows from (\ref{estimate_on_x0}) and (\ref{final_estimate3}).

The Figure~\ref{fig_example} demonstrates the efficacy of the proposed controller. For comparison, in the Figure~\ref{fig_example_dist} is shown the solution of the same control problem with the same values of the parameters,  $\tilde\lambda=10,$ $\Gamma_0=4$ and $T=1,$ to which it was added the time-varying disturbance term $w(t)=\left(0, 0.01(1-t)^{5/2}\cos 0.05t\right)^T.$ In this case, the inequality (\ref{hypothesisH2_1}) also gives the lower bound for $\omega,$ $\omega\in[1.382,2.382].$
\begin{figure}
   \centerline{
    \hbox{
     \psfig{file=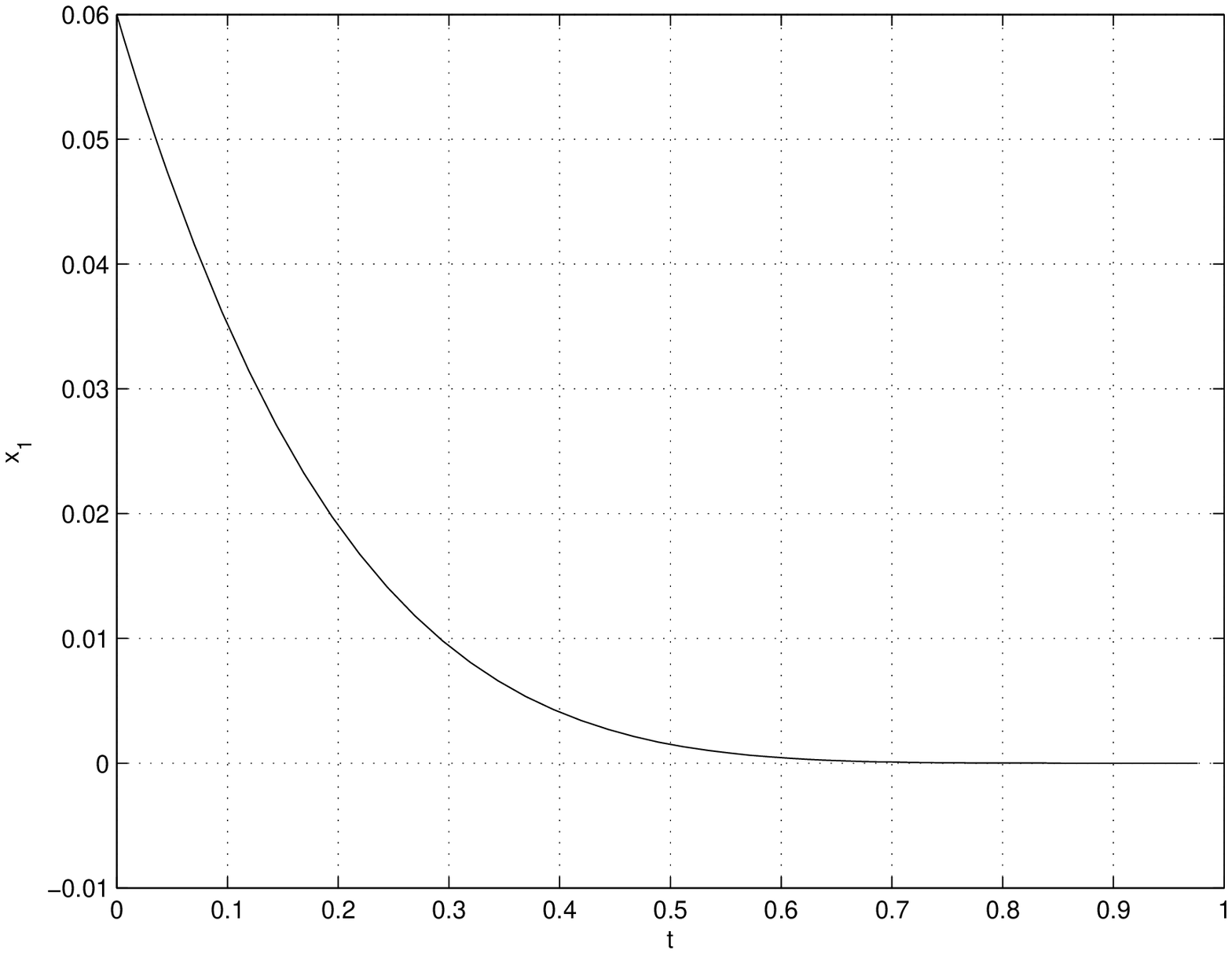,width=5cm, clip=}
     \hspace{1.cm}
     \psfig{file=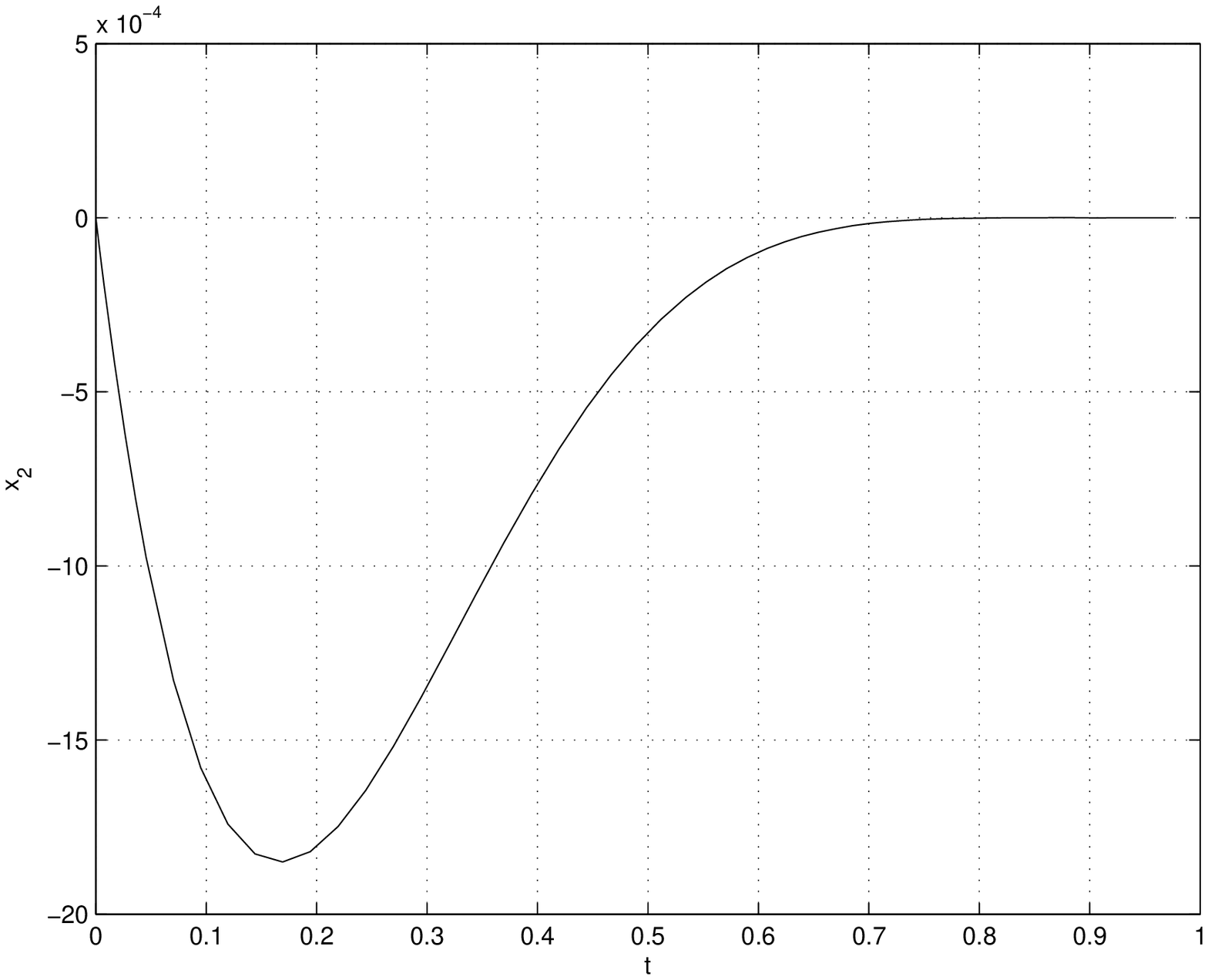,width=5cm,clip=}
    }
   }
      
\caption{The solution of control system (\ref{example_calc}) on the time-interval $[0,1]$ with the state feedback (\ref{controller_example}) for $\omega=2$ and $x_0=(0.06, 0)^T$ ($|x_0|<0.110497$); $\vert \dot x_-(1)\vert=0$ (a left-hand derivative) because $\lambda^{**}/\omega=2.382/2$ is greater than $1.$}
\label{fig_example}
\end{figure} 
\begin{figure}
   \centerline{
    \hbox{
     \psfig{file=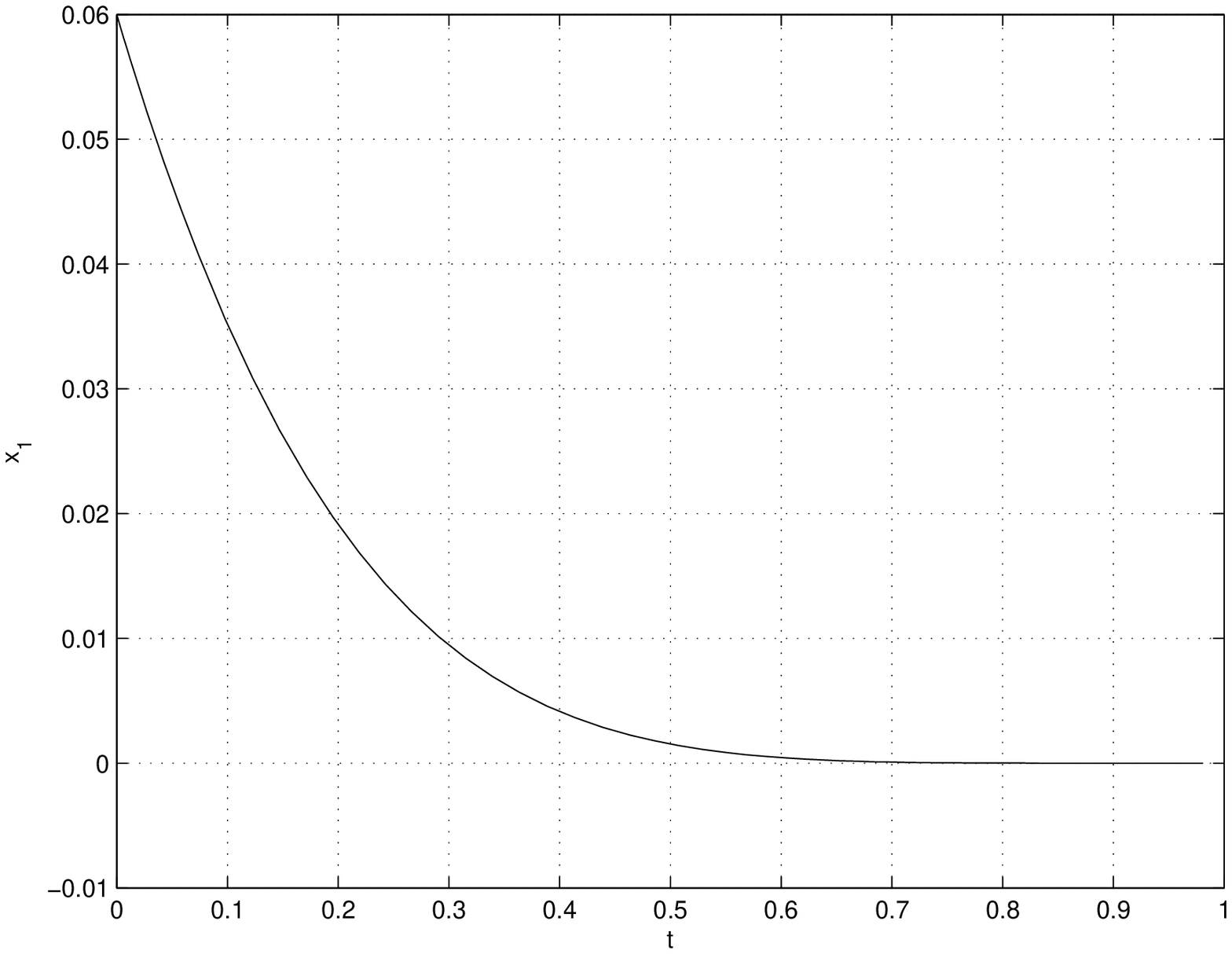,width=5cm, clip=}
     \hspace{1.cm}
     \psfig{file=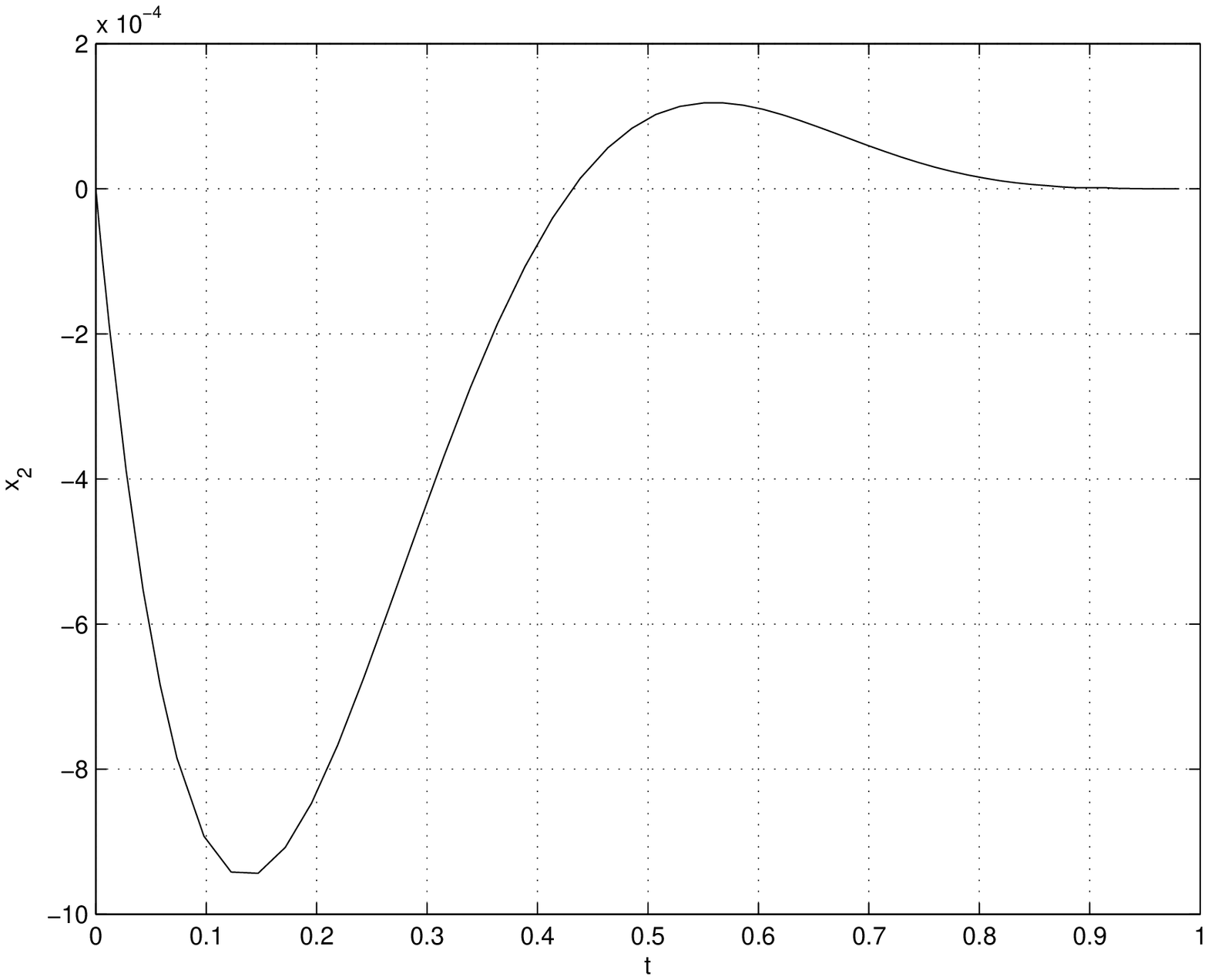,width=5cm,clip=}
    }
   }
     
\caption{The solution of control system (\ref{example_calc}) with the added disturbance term $w(t)=\left(0, 0.01(1-t)^{5/2}\cos 0.05t\right)^T$ on the time-interval $[0,1]$ with the state feedback (\ref{controller_example}) for $\omega=2$ and $x_0=(0.06, 0)^T$ ($|x_0|<\mu=0.07787$). For calculation of the radius $\mu$ we use the equation (\ref{ball}) with the parameters $k_1=1,$ $T=1,$ $M_d=0.01,$ $\gamma=-0.618,$ $\omega=2$ and $\varepsilon=0.110497.$ }
\label{fig_example_dist}
\end{figure} 
\end{ex}

\section[Generalization]{Generalization of the local null controllability to the other points as origin}

The origin is not essential for what has been derived previously, and we can also consider the controllability to the other points, different of origin. This section is just devoted to some generalization in that sense.

Denote by $\mathcal{LR}$ the set of all $x^*\in\mathbb{R}^n$ such that
\begin{itemize}
\item[(H1)$^*$] the pair $\left(A, B\right),$ $A^*=\left(\partial f/\partial x\right)(x^*),$ is controllable;
\item[(H2)$^*$] there exist the constants $\tilde\lambda^*>0,$ $\omega^*>0,$ $\delta^*<0$ and time $T^*$ such that 
\[
k_1^*(\Gamma_0^*+T^*||A^*||)+\omega^*\leq\tilde\lambda^*\leq k_1^*T^*||A^*||+\omega^*-\delta^* \ (\mathrm{for}\ M^*_d>0)
\]
or
\[
k_1^*(\Gamma_0^*+T^*||A^*||)+\omega^*\leq\tilde\lambda^* \ (\mathrm{for}\ M^*_d=0)
\]
\item[(H3)$^*$] $|f(x^*)+w(t)|\leq M^*_d\left(1-\frac{t}{T^*}\right)^{-\delta^*/\omega^*}$ for some $T^*=T^*(x^*)>0,$ and for all $t\in[0,T^*],$ where $\delta^*<-k_1^*\Gamma_0^*$ for $M_d^*>0$ and  $\delta^*=0$ if $M_d^*=0;$ 
\item[(H4)$^*$] the equation 
\[
k_1^*\mu^*+\left(e^{-\frac{k_1^*T^*M_d^*}{\gamma^*}}-1\right)\omega^*=\varepsilon^*
\]
has a positive solution $\mu^*.$
\end{itemize}
All parameters and constants have an analogous meaning as that in the proof of Lemma~\ref{lemma1}. Theorem~\ref{main_result} is embedded in its following natural generalization.
\begin{thm}\label{generalization}
For each $x^*\in\mathcal{LR}$ there exists an open neighborhood $\mathcal{X}_{x^*}$ of $x^*$ such that, to each $x_0\in\mathcal{X}_{x^*},$ there corresponds a continuous  function $u: [0,T^*]\rightarrow \mathbb{R}^m,$ such that the solution $x(t)$ of (\ref{def_system}) determined by this $u(t)=\left(\frac{1}{T^*-t}\right)\frac{K^*(x(t)-x^*)}{\omega^*}$ and $x(0)=x_0$ satisfies $x(T^*)=x^*.$
\end{thm}
\begin{proof}
Let us choose $x^*\in\mathcal{LR}.$ Defining the new state variable $\hat x$ by the relation $\hat x= x-x^*$ ($\hat u=u$), the original problem (\ref{def_system}) is transformed to $\dot{\hat x}=\hat f(\hat x)+B\hat u+w(t),$ $\hat x(0)=x(0)-x^*,$ where $\hat f(\hat x)=f(\hat x+x^*).$ Analyzing local null controllability of this new system analogously as in Lemma~\ref{lemma1}, Lemma~\ref{lemma2} and Theorem~\ref{main_result}, taking into account that $(\partial \hat f/\partial \hat x)(0)=(\partial f/\partial x)(x^*)$ there exists the bounded controller 
\[
\hat u=\hat v_\omega(\hat x,t)=\left(\frac{1}{T^*-t}\right)\frac{K^*\hat x}{\omega^*}, \ t\in[0,T^*],
\]
that is,
\[
u(t)=v_\omega(x(t)-x^*,t)=\left(\frac{1}{T^*-t}\right)\frac{K^*(x(t)-x^*)}{\omega^*}, \ t\in[0,T^*]
\]
for the original problem. Thus we have proved Theorem~\ref{generalization}.
\end{proof}
\begin{rmk}\label{remark:constrained}
\rm In the theoretical part of the paper as well as in the example we considered the control problem with unconstrained input $u(t),$  therefore in the case $M_d=0$ the null controllable region can be unlimitedly enlarged as follows from (\ref{ball}) and (\ref{example_ineq}) at the cost of enlarging $u(t)=v_\omega(x(t),t)$ at the same time since  $|u(t)|=O(\omega^{-1})$ for $\omega\to 0^+$ as we can see from (\ref{final_estimate3}) and (\ref{controller}) for $\lambda^{**}/\omega\geq1.$ But, in the case, when the control variable $u$ is the subject of the constraint of the form $|u(t)|\leq\Phi$ for all $t\in[0,T],$ in addition to the hypothesis (H2) of Lemma~\ref{lemma1}, for the parameter $\omega$ we obtain the sufficient condition to be $u(t)$ the admissible control
\[
|u(t)|\leq\frac{||K|||x(t)|}{(T-t)\omega}
\]
\[
\leq\frac{||K||}{(T-t)\omega}\left(1-\frac{t}{T}\right)^{\lambda^{**}/\omega}\left[k_1|x_0|+\left(e^{-\frac{k_1TM_d}{\gamma}}-1\right)\omega\right]
\]
\[
\leq\frac{||K||}{\omega T^{\lambda^{**}/\omega}}\left[k_1|x_0|+\left(e^{-\frac{k_1TM_d}{\gamma}}-1\right)\omega\right]\leq\Phi.
\]
The last inequality gives another bound to the (H2) for the parameter $\omega.$ 
\end{rmk}

\section{Remark on the global null controllability}
In this section we will determine the sufficient conditions to be the control system (\ref{def_system}) with $w(t)\equiv 0$ globally null controllable, that is, the set $\mathcal{X}_0$ in the Definitions~\ref{locally_null_c1} and \ref{locally_null_c2} is the whole state space, $\mathcal{X}_0=\mathbb{R}^n.$

Extracting the essence of Example~\ref{example1} we have the following key points:
\begin{itemize}  
\item[(a)] The boundedness on $\mathbb{R}^n$ of all second-order partial derivatives of $f$ implies that $\varepsilon$ can be made arbitrarily large by selecting a suitably small value of $\omega$ in the inequality (\ref{example_ineq});
\item[(b)] If $B$ is an invertible matrix, then $k_1=1$ \underline{independently} on $\tilde\lambda$ and which ensure that the assumptions (H1), (H2) and (H4) of Lemma~\ref{lemma1} are satisfied;  
\item[(c)] The equation for computing the null controllable region (\ref{ball}) reduces to $\mu=\varepsilon.$
\end{itemize}
Tus we have the following result on the global null controllability:
\begin{thm}
Let us consider the system (\ref{def_system}) with $w(t)\equiv 0,$  that is, the system $\dot x=f(x)+Bu$ with the initial state $x(0)=x_0.$ Let
\begin{itemize}
\item[(i)] all second-order partial derivatives of the function $f\in C^2$ are bounded on $\mathbb{R}^n;$
\item[(ii)] the matrix $B$ is regular.
\end{itemize}
Then the system under consideration is globally null controllable in the sense that for an arbitrary $T>0$ fixed and every $\mu=\varepsilon=\varepsilon(T,\Gamma_0,\omega)>0,$  there exists a control law of the form, which have been defined in Lemma~\ref{lemma1} with an appropriate feedback gain matrix $K$ such that $x(T)=0$ for all $x_0,$ $|x_0|<\mu.$
\end{thm}
Now we introduce an example demonstrating that if the assumption on the invertibility of the matrix $B$ is not fulfilled, the system may not be globally null controllable.
\begin{ex}\label{example2}
\rm Let us consider the following system:
\begin{equation}\label{example_calc2}
{\setlength\arraycolsep{2pt}
\left(
\begin{array}{c}
  \dot x_1  \\
  \dot x_2 
\end{array} 
\right)=
\left(
\begin{array}{c}
  x_1^2+\frac{x_2}{1+x_2^2} \\
   0
\end{array} 
\right)
+
\left(
\begin{array}{r}
  0\\
  1 
\end{array} 
\right)
u(t)
}.
\end{equation}

By applying the linearization of his system at $x=0$ we have
\begin{equation*}
{\setlength\arraycolsep{2pt}
A=\left(\partial f/\partial x\right)(0)=\left(\begin{array}{cc}
 0 \ & 1   \\
 0 \ & 0
\end{array} \right),\ 
B=
\left(
\begin{array}{r}
  0\\
  1 
\end{array} 
\right)
}.
\end{equation*}
Therefore  
\begin{equation*}
{\setlength\arraycolsep{2pt}
(B, AB)=
\left(
\begin{array}{cc}
 0 \ & 1   \\
 1 \ & 0
\end{array} \right)
}
\end{equation*}
which means the linearized system is controllable. However, considering the original nonlinear system, there is no controller which can move the states to zero in any time if the initial state $x_1(0)>1/2.$
\end{ex}

\section*{Conclusions}
 
In this paper, we have analyzed the local null controllability of the nonlinear control-affine systems of the form $\dot x(t)=f(x(t))+Bu(t)+w(t)$ with the time varying disturbances. To prove the main result and with purpose to ensure better estimates of the null controllable region, a new Gronwall-type inequality was derived. Under the hypothesis (H1)-(H4) we have shown in Lemma~\ref{lemma1} the existence of an open ball with radius $\mu>0$ contained in the null controllable region $\mathcal{X}_0$ around the origin and the state feedback control law steering the state of system from each initial state $x(0)=x_0\in\mathcal{X}_0$ to the origin at the finite time $T.$ Practical applicability the theory in explicit calculation of the null controllable region was documented on the examples with/without time varying disturbance. Subsequently, we have generalized notion of local null controllability to the more general type of local controllability,  where "null" can be replaced by any point $x^*$ from the set $\mathcal{LR}$  and finally a brief remark on the global null controllability of the system under consideration was given.

\end{document}